\documentclass[11pt]{amsart}

\usepackage{epigamath}

\usepackage[english]{babel}

\usepackage{mathrsfs}
\usepackage{array, boldline, makecell, booktabs}
\definecolor{gray}{gray}{0.4}
\usepackage{amssymb}
\usepackage[all]{xy}
\usepackage{multirow,multicol}

{\theoremstyle{definition}
\newtheorem{dfn}{Definition}[section]

\newtheorem{rmk}[dfn]{Remark}}
\newtheorem{pro}[dfn]{Proposition}
\newtheorem{lmm}[dfn]{Lemma}
\newtheorem{thm}[dfn]{Theorem}
\newtheorem*{thm*}{Theorem}

\newcommand{\bfx}{\mathbf{x}}
\newcommand{\mch}{\mathcal{H}}

\newcommand{\RR}{\mathbb{R}}
\newcommand{\QQ}{\mathbb{Q}}
\newcommand{\NN}{\mathbb{N}}
\newcommand{\ZZ}{\mathbb{Z}}
\newcommand{\IN}{\mathrm{in}}
\newcommand{\supth}[1]{#1^{\mathrm{th}}}

\EpigaVolumeYear{5}{2021} \EpigaArticleNr{7} \ReceivedOn{July 10, 2019}
\InFinalFormOn{November 19, 2020} \AcceptedOn{January 14, 2021}

\title{Bia\l ynicki-Birula schemes in Hilbert schemes of points \\ and monic functors}
\titlemark{Bia\l ynicki-Birula schemes in Hilbert schemes of points}

\author{Laurent Evain}
\address{Universit\'e d'Angers,
Facult\'e des Sciences,
D\'epartement de math\'ematiques,
2, Boulevard Lavoisier,
49045 Angers Cedex 01,
France}
\email{laurent.evain@univ-angers.fr}

\author{Mathias Lederer}
\address{Department of Mathematics,
University of Innsbruck,
Technikerstrasse 19a,
A-6020 Innsbruck,
Austria}
\email{mlederer@math.cornell.edu}

\authormark{L. Evain and M. Lederer}

\AbstractInEnglish{The Bia\l ynicki-Birula strata on the Hilbert scheme
$H^n(\mathbb{A}^d)$ are smooth in dimension $d=2$.
We prove that there is a schematic structure in higher dimensions,
the Bia\l ynicki-Birula scheme, which is natural in the sense that
it represents a functor. Let $\rho_i:H^n(\mathbb{A}^d)\rightarrow
{\rm Sym}^n(\mathbb{A}^1)$ be the Hilbert-Chow morphism of the $\supth{i}$
coordinate. We prove that a Bia\l ynicki-Birula scheme associated with an
action of a torus $T$ is schematically included  in the fiber
  $\rho_i^{-1}(0)$ if the $\supth{i}$ weight of $T$ is non-positive. We
  prove that the monic functors parametrizing families of ideals with
  a prescribed initial ideal are representable.}

\MSCclass{14C05; 13P10; 14Q99; 14R20}

\KeyWords{Hilbert scheme of points; standard sets; Bia{\l}ynicki-Birula decomposition;
stratification; torus action; representable functor; Gr\"obner basis}

\TitleInFrench{Sch\'emas de Bia\l ynicki-Birula dans les sch\'emas de Hilbert de points et foncteurs moniques}

\AbstractInFrench{Les strates de Bia\l ynicki-Birula sur le sch\'ema de Hilbert
$H^n(\mathbb{A}^d)$ sont lisses en dimension $d=2$.
Nous d\'emontrons qu'il existe une structure sch\'ematique en dimension sup\'erieure,
le sch\'ema de Bia\l ynicki-Birula, qui est naturelle au sens o\`u elle repr\'esente un foncteur. Consid\'erons $\rho_i:H^n(\mathbb{A}^d)\rightarrow
{\rm Sym}^n(\mathbb{A}^1)$ le morphisme de Hilbert-Chow de la $i^{\mbox{\scriptsize \`eme}}$
coordonn\'ee. Nous prouvons qu'un sch\'ema de Bia\l ynicki-Birula associ\'e \`a l'action
d'un tore $T$ est sch\'ematiquement inclus  dans la fibre
  $\rho_i^{-1}(0)$ si le $i^{\mbox{\scriptsize \`eme}}$ poids de $T$ est n\'egatif ou nul. Nous prouvons que les foncteurs moniques qui param\`etrent les familles d'ideaux ayant un id\'eal initial donn\'e
 sont repr\'esentables.}

\acknowledgement{The second author was supported by a Marie Curie International Outgoing Fellowship
of the EU Seventh Framework Program}

\begin{document}


\removeabove{0.5cm} \removebetween{0.5cm} \removebelow{0.5cm}

\maketitle

\begin{prelims}

\DisplayAbstractInEnglish

\bigskip

\DisplayKeyWords

\medskip

\DisplayMSCclass

\bigskip

\languagesection{Fran\c{c}ais}

\bigskip

\DisplayTitleInFrench

\medskip

\DisplayAbstractInFrench

\end{prelims}


\newpage

\setcounter{tocdepth}{1}

\tableofcontents


\section{Introduction}\label{intro}

Let $H^n(\mathbb{A}^d)$ be the Hilbert scheme
parametrizing zero-dimensional subschemes of
length $n$ in the affine $d$-space  $\mathbb{A}^d$ over a field $k$.
This scheme is mostly called the \emph{Hilbert scheme of points},
sometimes also the \emph{punctual Hilbert scheme}.
There is a natural action of the $d$-dimensional split torus on $\mathbb{A}^d$,
which induces a natural action on $H^n(\mathbb{A}^d)$.
If $T$ is a one-dimensional split subtorus
of the $d$-dimensional torus,
then $T$ defines the
\emph{Bia\l ynicki-Birula strata} $H^{BB(T,\Delta)}$
parametrizing the subschemes of $\mathbb{A}^d$ converging
to some fixed point $Z^\Delta$ under the action of $T$,
where $Z^\Delta$  is a monomial subscheme with staircase
$\Delta$. When $T$ is
general, any $T$-fixed point is monomial and is a $Z^\Delta$ for some
staircase $\Delta$. The case where $T$ is general is thus of
particular interest, however we will consider   $H^{BB(T,\Delta)}$ for
any $T$.

These stratifications are preeminent in most studies of the punctual Hilbert
scheme in dimension two. For instance, they
appear in the computation of the Betti
numbers (see \cite{esBetti}, \cite{esCells}), in the determination of
the irreducible components of (multi)graded Hilbert schemes
(see \cite{evainIrred}, \cite{MaclaganSmith}),
or in the study of the ring of symmetric functions
via symmetric products of embedded curves (see \cite{grojnowski}, \cite{nakajima}).

The Bia\l ynicki-Birula strata in $H^n(\mathbb{A}^2)$ are  affine spaces.
In contrast, not much is known on these  strata
for higher dimensional $\mathbb{A}^d$, and the difficulty to control
and describe the Bia\l ynicki-Birula
strata is probably one of the reasons why the Hilbert scheme of points is
still mysterious in higher dimensions.

In dimension three, the Bia\l ynicki-Birula strata are not
irreducible (Proposition \ref{BBmayBeIrreducible}). In higher
dimensions, they are not
reduced either \cite{jel2}.
It is therefore necessary to define them with
their natural scheme structure as representing a functor.
Apart from the necessity to define them schematically,
it is desirable to have functorial descriptions of Hilbert schemes at hand,
as these descriptions are known to be both powerful and easy to handle.

In the present paper,
we introduce the Bia\l ynicki-Birula functor parametrizing families of
subschemes $Z$ such that $\lim_{t\rightarrow 0,t\in T}t\cdot Z=Z^{\Delta}$ for
some fixed monomial subscheme $Z^\Delta$. We will prove (Theorem \ref{thm:BBrepresentable}):
\begin{thm*}
  The Bia\l ynicki-Birula functor is representable by a locally closed subscheme $H^{BB(T,
    \Delta)}(\mathbb{A}^d)$ of the Hilbert scheme $H^n(\mathbb{A}^d)$.
\end{thm*}
The theorem is
constructive: if a flat family is given, the Bia\l
ynicki-Birula strata are computable by the algorithms
encapsulated in the proofs.

During the proof, a linchpin
construction is to consider ideals with a prescribed initial ideal
for a total order. The total order considered is any order
refining the partial order on the monomials induced by
their  weight for the torus action.
We prove the representability of the
corresponding functor
(Theorem
\ref{thm:theMonicFuctorIsrepresentable}).
\begin{thm*}
  The monic functor $ \mch^{{\rm mon}(<,\Delta)}$
  parametrizing ideals with initial ideal $I^\Delta=I(Z^\Delta)$ is representable.
\end{thm*}
There are results in the same circle of ideas when Gr\"obner
basis theory is workable \cite{lederer,lella2016}.
In the context of Theorem \ref{thm:theMonicFuctorIsrepresentable},
Gr\"obner basis theory does not apply
because of possible negative weights. As far as we know, monic
functors have never been considered in this context. We develop an original
approach for the proof as the ideas  from   \cite{lederer,lella2016} are not
easily adjustable.

Our sign convention for the weight
$\xi=(\xi_1,\dots,\xi_d)$ is that
the action
 of the one-dimensional subtorus $T$
on $\mathbb{A}^d$ is given
by $t\cdot (a_1, \ldots, a_d) :=
(t^{\xi_1}a_1,\dots,t^{\xi_d}a_d)$.
If $\xi_i\leq 0$,
then the closed points of $H^{BB(T,\Delta)}(\mathbb{A}^d)$
correspond to subschemes $Z$ whose support is in the hyperplane $x_i=0$.
This follows from the naive observation
that if $t\cdot Z$ tends to $Z^\Delta$, then the support of $t.Z$ tends to
the support of $Z^\Delta$. A much more subtle question is to ask
whether this remains true at the
schematic level, when we consider the Bia\l ynicki-Birula scheme with its
possibly non-reduced structure. The answer is positive.
Recall that there is a Hilbert-Chow morphism
$\rho_i:H^n(\mathbb{A}^d)\rightarrow {\rm Sym}^n(\mathbb{A}^1)\simeq \mathbb{A}^n$
which sends a point $p$ parametrizing a subscheme $Z$ to
(the coefficients of) the characteristic polynomial of the
multiplication by the $\supth{i}$ coordinate $x_i$ in $\mathcal{O}(Z)$.
We will prove
(see Theorem \ref{thm:HilbertChow}):
\begin{thm*}
  Let $\rho_i$ be the Hilbert-Chow morphism associated with the $\supth{i}$
  coordinate.  If $\xi_i\leq 0$, then  $H^{BB(T,
    \Delta)}(\mathbb{A}^d)$ is schematically included in the fiber
  $\rho_i^{-1}(0)$.
\end{thm*}

For simplicity, we have considered a field $k$ in this introduction.
But throughout the paper, we shall work over  a ring $k$ of arbitrary
characteristic. \medskip

After the first version of this article appeared, several authors dealt with the
Bia\l ynicki-Birula decomposition for singular schemes, algebraic
spaces, stacks...
with methods, scope, computability specific to each approach
and important applications
\cite{drinfeld,ahr,richarz,jel1,jel2,kambe}.

Let Z be an algebraic $k$-space of finite type equipped with a
$\mathbb G_m$-action. Drinfeld \cite{drinfeld} defines the
attractor $Z^+$, which informally is the functor whose points
are the points $z\in Z(k)$ with an existing limit $\lim_{t\rightarrow 0}t\cdot z$.
He proves that $Z^+$ is representable (Theorem~1.4.2, \emph{ibid.}).
Using the results by Drinfeld, it is possible to recover
some results of the present paper (see Section
\ref{sec:relationsToLitterature}).

Alper-Hall-Rydh \cite{ahr} work in the context of algebraic stacks.
They prove (Theorem~5.27, \emph{ibid.}) the existence of Bia{\l}ynicki-Birula decomposition for Deligne-Mumford stack of finite type over a field $k$, equipped with a $\mathbb{G}_m $-action.
They recover some results by Drinfeld (\emph{cf.} Remark~5.28, \emph{ibid.}).
\medskip

Richarz extends some representability
results by Drinfeld to algebraic spaces with an \'etale locally
linearizable $\mathbb{G}_ m$-action \cite[Theorem~A]{richarz}
\medskip

An important application of the Bia\l ynicki-Birula decomposition
is due to Jelisiejew. In \cite[Proposition~3.1]{jel1},
he realizes the Bia\l ynicki-Birula
decomposition in the multigraded Hilbert scheme of Haiman-Sturmfels
\cite{hs}. Then, in \cite{jel2},
the decomposition is used to prove
that the Hilbert scheme of points on a higher dimensional affine space
is non-reduced and has components lying entirely in characteristic $p$
for all primes $p$, and
that Vakil's Murphy's Law (every singularity type of finite type appears) holds up to retraction for this scheme.
\medskip

\noindent
\textbf{Proofs.} Let us say a word about the proofs. The action of $T$ on
$\mathbb{A}^d={\rm Spec}\, k[\bfx]={\rm Spec}\; k[x_1,\dots,x_d]$ with weight $\xi$
induces a partial order $<_\xi$
on the monomials of $k[\bfx]$: the
monomials are ordered according to their weight  for the $T$-action.
To control the Bia\l ynicki-Birula strata $H^{BB(T,\Delta)}$
we prove that, roughly speaking,  a subscheme $Z$ is in
$H^{BB(T,\Delta)}$ if and only if its initial ideal $\IN_{<_\xi}(I(Z))$ equals
the monomial ideal $I(Z^\Delta)$
(Proposition \ref{reformulationBB2Monic}). Therefore,
a natural strategy could be to introduce a monic functor
$ \mch^{{\rm mon}(<_\xi,\Delta)}$
parametrizing
ideals with a prescribed initial ideal and to show that this functor
is representable and isomorphic to the Bia\l ynicki-Birula functor.

However, a technical barrier is that the
initial ideal in our context is a poor analog of the same notion used in the context
of Gr\"obner basis, for instance in \cite[Chapter~15]{eisenbud} for two
reasons: the order on the monomials is a only a
\emph{partial }order, and when some weights of the action are negative, the
division algorithm may not terminate. This makes the above strategy
inefficient.

The modified strategy is the following.  We consider \emph{ total
  orders} rather than \emph{partial } orders when we introduce the monic
functors mainly because this condition
ensures functoriality (Remark \ref{rem:functorialityNeedsTotalOrder}).
However, because of the possible negative weights, we still don't
have  in general a monomial order  in the sense of  \cite{eisenbud}.
We prove that the monic functor
$ \mch^{{\rm mon}(<,\Delta)}$ parametrizing
ideals with initial ideal $I^\Delta$ is representable
(Theorem \ref{thm:theMonicFuctorIsrepresentable}).

In this modified strategy,  we realize the Bia\l ynicki-Birula
functor $\mch^{BB(T,\Delta)}$ as the intersection of two well-chosen monic functors
$\mch^{{\rm mon}(<_-,\Delta)}\cap \mch^{{\rm mon}(<_+,\Delta)}$
where $<_-$ and $<_+$ are total orders refining
$<_\xi$ (Proposition
\ref{BBFunctor=intersectionOfTwoMonicFunctors}).
Realizing the Bia\l ynicki-Birula functor as an intersection
is the functorial counterpart to the following remark :
having a prescribed initial ideal for the
partial order $<_\xi$ is equivalent to having the same prescribed initial
ideal for both $<_-$ and $<_+$.
 The representability
of the  Bia\l ynicki-Birula
functors then follows from the representability of the monic functors.


When dealing with representability of functors, constructions for
individual subschemes often require uniformity lemmas when one passes
to families. The paragon of this situation is the Castelnuovo-Mumford regularity,
in the construction of the Hilbert scheme. In contrast,
punctual subschemes localized at a fixed point $p$ of the projective
space are not representable by a closed subscheme of the Hilbert
scheme because the families lack a uniformity property : The smallest
infinitesimal neighborhood of $p$ containing
a family of such subschemes may be of arbitrary large order.

When the weights of the action are non-positive, the families that we consider are supported on
the origin. It may be surprising that we get the representability
in this context. The reason is that Bia\l ynicki-Birula families
are included in an infinitesimal neighborhood of uniform order.
This uniformity is settled in Lemma \ref{propSecondFinitnessLemma}
which says that the families of the Bia\l ynicki-Birula functor are
included in $x_i^n$ if $x_i$ is a coordinate with a non-positive weight,
where $n$ is the length of the parameterized punctual subschemes.
In some sense, these families are uniformly bounded. As to the monic
functors, some
boundedness condition is included in their definition,
\emph{i.e.} the families are included in $x_i^r$ for some $r$,
and one proves that $r$ may be chosen uniform
in Proposition \ref{pro:uniformBoundingForMonic}
to get the representability.


\subsection*{Acknowledgements}
We give many thanks to Robin Hartshorne, Diane Maclagan and Gregory G. Smith,
the organizers of the workshop {\it Components of Hilbert Schemes}
held at American Institute of Mathematics in Palo Alto, CA, in July 2010.
It was there that we first met and shared our thoughts;
the exceptionally productive and at the same time friendly atmosphere of that workshop triggered our collaboration.
Special thanks go to Bernd Sturmfels, who realized that we have many research interests in common,
and motivated us to work together.
The second author wishes to thank Allen Knutson, Jenna Rajchgot, and Mike Stillman for many fruitful discussions.
We thank the anonymous referee for useful comments.


\section{Bia\l ynicki-Birula functors and $\Delta$-monic families}
\label{sec:bialyn-birula-funct}

In this section, we introduce the Bia\l ynicki-Birula functors whose
geometric points are subschemes having a prescribed limit under the action of
a torus $T$. These subschemes are characterized by their
initial ideal for some order. Accordingly,
we  reformulate Bia\l ynicki-Birula families in terms of $\Delta$-monic
ideals (Proposition \ref{reformulationBB2Monic}).
Finally, we prove
an important uniformity lemma for Bia\l ynicki-Birula families (Corollary \ref{grobnerFamiliesAreBounded}).
\smallskip

In this paper, we consider schemes over a commutative ring $k$ of
arbitrary characteristic. We denote by $k[\bfx]$ the polynomial
ring $k[x_1,\dots,x_d]$. Similarly, for $e=(e_1,\dots,e_d) \in \NN^d$,
we use the multi-index notation $\bfx^e:=x_1^{e_1}\cdots x_d^{e_d}$.

A \emph{standard set}, or \emph{staircase}, is a subset $\Delta \subset \NN^d$ whose
complement $C := \NN^d\setminus \Delta$ satisfies $C+\NN^d = C$.
We call the minimal elements of $C$ the \emph{outer corners of $\Delta$.}
All standard sets under consideration will be of finite cardinality $n$.
The ideal generated by the monomials $\bfx^e,e\in C$ is denoted
by $I^\Delta$. The notation  $I^\Delta$ makes sense in $k[\bfx]$,
but more generally in any ring containing the
monomials $\bfx^e$, such as the ring $B[\bfx]$ introduced below.
We shall freely identify the monomials $\bfx^e$ with
their exponent $e$. In particular, the notion of a staircase of monomials makes sense.

If $B$ is a $k$-algebra, then the tensor product $B\otimes_k k[\bfx]$
is just $B[\bfx]$, the ring of polynomials with coefficients in $B$.
Similarly, we write $B[t,t^{-1},\bfx] := B[\bfx]\otimes_k
k[t,t^{-1}]$. Let $\xi\in \RR^d$, we denote by
$f_\xi\in (\RR^d)^*$ the linear form defined by $f_\xi(\alpha_1,\dots,\alpha_d)=\sum
  \alpha_i \xi_i$. If $\xi\in \ZZ^d$, and $I\subset B[\bfx]$,
$I_\xi[t,t^{-1}]\subset B[t,t^{-1},\bfx]$ denotes the ideal generated
by the elements $t\cdot g:= \sum t^{-f_\xi(e)}c_e\bfx^e$ where
$g=\sum c_e \bfx^e\in I$.
We denote by $I_\xi[t]\subset B[t,\bfx]$ the ideal
$I_\xi[t,t^{-1},\bfx]\cap B[t,\bfx]$. When $a\in k$, we denote by
$I_\xi[a]\subset B[\bfx]$ the ideal $\phi_a(I_\xi[t])$ where
$\phi_a:B[t,\bfx]\rightarrow B[\bfx]$ is the evaluation morphism
sending $t$ to $a$. In particular $I=I_\xi[1]$.

\begin{dfn}
 We denote by
  $\mch^{BB(\Delta,\xi)}(B)$, or more simply by
  $\mch^{BB(\Delta)}(B)$ when $\xi$ is obvious,
  the set of ideals $I\subset  B[\bfx]$
  such that $\lim_{t \rightarrow 0} t\cdot I=I^\Delta$, which means:
  \begin{itemize}
  \item $B[t,\bfx]/I_\xi[t]$ is a locally free $B[t]$-module of rank $n = \#\Delta$.
  \item $I_\xi[0]=I^\Delta\subset B[\bfx]$
  \end{itemize}
  $\mch^{BB(\Delta)}$ is a covariant functor from the category of
  $k$-algebras to the category of sets. We call it the \emph{Bia\l ynicki-Birula functor}.
\end{dfn}

Remark that the sign convention for the action
is consistent with the choices made in the
introduction. The action of the torus was $t\cdot (a_1, \ldots, a_d) :=
(t^{\xi_1}a_1,\ldots,t^{\xi_d}a_d)$
on $\mathbb{A}^d$.
This corresponds to the torus action on the polynomial ring $k[\bfx]$ which is trivial on scalars
  and is given by $t.\bfx^e := t^{-f_\xi(e)}\bfx^e$ on monomials.

The first bulleted item of the definition says that
${\rm Spec}\,B[t,\bfx]/I_\xi[t] \to {\rm Spec}\,B[t]$ is a finite flat family.
The second bulleted item says that its fiber over $t=0$
is the monomial subscheme of ${\rm Spec}\,B[t,\bfx]/<t>={\rm Spec}\,B[\bfx]$ defined by $I^\Delta$.
The limit $\lim_{t \rightarrow 0} t\cdot I$ is therefore a well-defined flat limit.


\begin{dfn} \label{defn:order}
  Let $\xi \in \RR^d$. We define the partial order $<_\xi$ on
  monomials in $k[\bfx]$ by setting $\bfx^e<_\xi\bfx^g$ if
  $f_\xi(e)<f_\xi(g)$, and letting $\bfx^e$ and $\bfx^g$ be incomparable if
  $f_\xi(e)=f_\xi(g)$. Since we identify monomials and exponents,
  we adopt the convention $f_\xi(\bfx^e):=f_\xi(e)$.
  \\
  A variable $x_i$ is called positive (resp. negative, non-positive,
  non negative) for $<_\xi$ if $\xi_i>0$ (resp $\xi_i<0$, $\xi_i\leq 0$, $\xi_i
  \geq 0$).
  \end{dfn}

If the weights $\xi_i$ are linearly independent
over $\QQ$, then $<_\xi$ is a total order on monomials.
Otherwise, the order is only partial.
The most interesting case for us  is when $\xi\in \ZZ^d$ is the weight
vector of the action.

\begin{dfn}
\leavevmode
  \begin{itemize}
  \item Let $\xi \in \RR^d$,  $<_\xi$ be the associated order,  and
  $f=\sum a_{e} \bfx^e\in B[\bfx]$.
  Let $\bfx^{e_1},\ldots,\bfx^{e_l}$ be the maximal monomials appearing
  in $f$ (with non-vanishing coefficients $a_{e_i},\,i=1\ldots l$). The \emph{initial form of $f$} for $<_\xi$ is $\IN(f):=\sum
  a_{e_i}\bfx^{e_i}$. This is a unique term when $<_\xi$ is a total order,
  but may be a sum of terms otherwise. We denote
  by $\IN(I)$ the ideal generated by the elements $\IN(f)$, $f\in I$.

  \item Let $I\subset B[\bfx]$ and $m=\bfx^e$ be a monomial. We denote by
  $\IN_m(I)\subset B$ the ideal generated by the elements $b\in B$ such that
  $\IN(f)=bm$ for some
  $f\in I$.

  \item
  Let $\Delta$ be a standard set of cardinality $n$.
  The ideal $I$ is called $\Delta$-monic
  if $\IN_m(I)=\langle 1 \rangle$ for all $m\notin \Delta$ and
  $\IN_m(I)=0$ for all $m\in \Delta$.
  \end{itemize}
\end{dfn}

The following proposition connects
Bia\l ynicki-Birula families and the order
$<_\xi$: an element $g$ is in the limit ideal
$I[0]=\lim_{t \rightarrow 0}t\cdot I$ iff its homogeneous components $g_i$ are
initial parts of elements of $I$.
\begin{pro}\label{limitAndOrder}
  Let $I\subset B[\bfx]$. Let $g=\sum_i g_i \in B[\bfx]$,
  with $g_i=\sum_{g_\xi(e)=i}c_e \bfx^e$. Then the following
  conditions are equivalent:
  \begin{enumerate}
\item $g\in I_\xi[0]$
\item $\forall i,\, \exists l_i\in I,\, \IN(l_i)=g_i$ for the order
  $<_\xi$.
    \end{enumerate}
\end{pro}
\begin{proof}
  To prove $2\Rightarrow 1$, let $s_i:=t^i(t\cdot l_i)$. We have $s_i\in
  I_\xi[t]$ and $s_i(0)=\IN(l_i)=g_i$. Thus $g_i\in I_\xi[0]$ and
  $g=\sum g_i \in I_\xi[0]$
  too.\\
  Conversely, if $g\in I_\xi[0]$, then $g=h(0)$ with $h\in
  I_\xi[t]$. We decompose $h$ as
  $h=\sum_j s_j$ with $s_j=P_j(t\cdot l_j)$ for some $l_j \in I$ and
  $P_j\in B[t,t^{-1},\bfx]$. By linearity, we may suppose that $P_j=c_j
  \bfx ^{a_j}t^{b_j}$. Since
  $P_j(t\cdot l_j)=t^{b_j+f_\xi(a_j)}(t\cdot (c_j\bfx^{a_j}l_j))$ we may suppose
  that $a_j=0$ and $c_j=1$. Grouping  terms by linearity, we may suppose that
  the exponents $b_j$ are pairwise distinct. If we put the weights on the
  variables with the formulas $deg(\bfx^e)=f_\xi(e)$ and
  $deg(t)=1$, then $s_j=t^{b_j}(t\cdot l_j)$ is the only term with degree
  $b_j$ in the sum $h=\sum_j s_j$. Since $h\in B[t,\bfx]$ it follows that
  $s_j\in B[t,\bfx]$ and $b_j\geq f_\xi(e)$ for every term $c_e \bfx^e$ in
  $l_j$. If $b_j> f_\xi(e)$ for every term, then $s_j(0)=0$ and we may
  ignore $s_j$ in the definition of $g=\sum s_j(0)$. If the equality
  occurs for some term, then $s_j(0)=\IN(l_j)$. This shows that the
  homogeneous components of $g$ are initial terms.
\end{proof}

\begin{lmm} \label{lm:rank_condition}
 Let $A$ be a commutative ring, $M$ be an $A$-module, $n\ge1$ be an integer and $f:A^n \oplus M
  \rightarrow A^n$ be an 
injective   morphism of $A$-modules. Then $M=0$.
\end{lmm}
\begin{proof}
First, we suppose that $A$ is Noetherian. Let $M^{
    i}$ be the direct sum $M\oplus \dots \oplus M$ with $i$ copies of
  $M$. Let $f_i:A^n\oplus M^{ i+1}\rightarrow A^n\oplus M^{i}$
 be  the morphism defined by
 $f_i(a,m_0,m_1,\dots,m_i)=(f(a,m_0),m_1,\dots,m_i)$.
In particular
  $f_0=f$. Let $c_i=f_0\circ f_1 \circ \dots \circ f_i: A^n\oplus
  M^{i+1}\rightarrow A^n$. Since $f$ is injective, $f_i$ and $c_i$ are
  injective for all $i\ge 0$.
  Let $M^i_{i+1}\subset M^{i+1}$ be the set of
  elements $(m_0,\dots,m_i)$ with $m_0=0$.  If $M$ is not zero, the
  inclusion
  $M^i_{i+1} \subset  M^{i+1} $ is strict .  The
  inclusions $0\oplus M^i=f_i(0\oplus M_{i+1}^i)\subset
  f_i(0\oplus M^{i+1})$, and $c_{i-1}(0\oplus M^i )\subset c_{i-1}\circ
  f_i(0 \oplus M^{i+1})=c_{i} (0 \oplus
  M^{i+1})$ are strict too by injectivity. Thus $M=0$, otherwise we would have a non stationary
  increasing sequence $c_{i-1}(0\oplus M^i)$ of submodules in the
  Noetherian module $A^n$.
  
  To prove the general case, let $N$ be the matrix
  of $g:A^n\rightarrow A^n$, where  $g$ is the restriction of $f$ to
  $A^n\oplus 0$, and let $m$ be any element of $M$.
  Let $B$ be the $\ZZ$-subalgebra of $A$ generated by
  the entries of $N$ and the entries of the vector $f(0,m)$. Let
  $h:B^n\oplus Bm \rightarrow B^n$ be the morphism of $B$-modules obtained by
  restriction of $f$. Since $B$ is Noetherian and $h$ is
  injective, we get $m=0$. 
\end{proof}

\begin{lmm}\label{lm:cobase}
  Let $I\subset B[\bfx]$ be an ideal with $\IN(I)=I^\Delta$ for the
  order $<_\xi$ and
  $B[\bfx]/I$ free of rank $\# \Delta$ as a $B$-module. Then the
  monomials $\bfx^e$, $e\in \Delta$ form a basis of $B[\bfx]/I$.
\end{lmm}
\begin{proof}
  The
  monomials $\bfx^e$, $e\in \Delta$  are independent mod $I$
  since $\IN(I)=I^{\Delta}$. We argue by contradiction and we  suppose
  that they generate
  a strict $B$-submodule $C\subset B[\bfx]/I$. 
  Then there exists  $f\in B[\bfx]$ whose class $\dot f \in B[\bfx]/I$ satisfies $\dot
  f \notin C$.
Let $B[\bfx^\Delta]\subset
  B[\bfx]$ be the $B$-submodule generated by the monomials $\bfx ^e, e\in
  \Delta$.
Replacing $f$ by
  $f-h$, $h\in B[\bfx^\Delta]$, one may suppose that $f$ has no terms
  in $\Delta$.  In particular, there exists $g\in I$, with
  $\IN(g)=\IN(f)$.  Replacing $f$ by $f-g$ lowers the initial form of
  $f$.   Repeating the process of killing the terms of $f$ in
  $\Delta$ and of
lowering the initial form $\IN(f)$,
one may suppose that
    $\IN(f)$ 
  is smaller than any monomial in the finite set $\Delta$.
  It follows that any term $\lambda_e \bfx^e$ in $f=\sum
  \lambda_e \bfx^e$ is smaller than any monomial in  $\Delta$ too, hence
  not in $\Delta$.
  In particular $f\in I^\Delta$ and
  there is a direct sum $$B f \oplus  B[\bfx^\Delta] \subset
  B [\bfx]=I^\Delta\oplus B[\bfx^\Delta].$$
    Since $\IN(I)=I^\Delta$, the intersection of $I$ with
  $B f \oplus  B[\bfx^\Delta]$ contains only elements of the form
  $bf\oplus 0$. Quotienting the displayed inclusion by $I$
  yields an injection $Bf/(I\cap Bf) \oplus B[\bfx^\Delta] \rightarrow
  B[\bfx]/I$.  This contradicts Lemma \ref{lm:rank_condition}
since $\mathrm{rank}( B[\bfx]/I)=\mathrm{rank}(B[\bfx^\Delta])=\#\Delta$.
  \end{proof}

\begin{pro} \label{reformulationBB2Monic}
  Let $\xi\in \ZZ^d$ and $<_\xi$ the associated partial order.
  Then $I\in \mch^{BB(\Delta,\xi)}(B)$ if, and only if, the following conditions are satisfied:
  \begin{enumerate}
    \item $\IN(I)=I^\Delta$
  \item $B[\bfx]/I$ is a locally free $B$-module of rank $\#\Delta$.
  \end{enumerate}
\end{pro}

\begin{proof}
  We only prove that the conditions imply that $I\in \mch^{BB(\Delta,\xi)}$,
  the converse being easy using Proposition~\ref{limitAndOrder}.

  Let us temporarily assume that $B[t,\bfx]/I_\xi[t]$ is a locally free $B[t]$-module.
  Then the flat limit $\lim_{t \rightarrow 0} t\cdot I$ exists
  and Proposition \ref{limitAndOrder} implies that this limit is
  $I^\Delta$.   For proving that $I\in \mch^{BB(\Delta,\xi)}(B)$,
  it therefore remains to show that
  $P[t] := B[t,\bfx]/I_\xi[t]$ is a locally free $B[t]$-module.

  Upon localizing $B$ (that is to say, upon replacing ${\rm Spec}\,B$
  with an affine open subset),
  we may assume that $B[\bfx]/I$ is $B$-free of rank $\#\Delta$.

  The monomials $\bfx ^e, e\in \Delta$ are a basis
  of $B[\bfx]/I$ by Lemma \ref{lm:cobase}.
  They are therefore also a basis of the $B[t,t^{-1}]$-module
  $B[\bfx]/I\otimes_k k[t,t^{-1}]\simeq B[\bfx,t,t^{-1}]/I_\xi[t,t^{-1}]$
  the latter isomorphism being given by the torus action $t\cdot \bfx^e =
  t^{-\xi\cdot e}\bfx^e$ on the polynomial ring.

  The monomials  $\bfx ^e$, for $e\in \Delta$,
  remain linearly independent in
  the $B[t]$-module $P[t]$. It remains to
 prove that $P[t]$ is a finite $B[t]$-module generated by the
  elements $\bfx^e, e\in \Delta$. It
  suffices to exhibit for every monomial $m\in B[\bfx]$, $m\notin\Delta$  a polynomial $P=m+\sum_{e
    \in \Delta}c_e\bfx^e$ with $P\in I_\xi[t]$. Since the elements
  $\bfx^e,e\in \Delta$ form a basis of $B[\bfx]/I$, the decomposition
  of $m$ yields an expression $m=\sum c_e\bfx^e\
  \mod I$.
  By condition 1), $c_e=0$ if $x^e \geq_\xi m$.  Thus we can take
  $P=t^{f_\xi(m)}(t\cdot (m-\sum c_e\bfx^e))$.
\end{proof}

\begin{rmk} According to the intuition from Gr\"obner bases, one could think
  that the second condition is a consequence of the first one. This is not the case, as is shown by
  the example $d=1$, $I=\langle x^2+x^3 \rangle$, $\xi=-1$, $\Delta=\{1,x\}$.
\end{rmk}

\begin{dfn}
  \label{bounded_ideal}
  Let $\xi \in \RR^d$ and $<_\xi$ the corresponding order.
  An ideal  $I\subset B[\bfx]$ with $B[\bfx]/I$ finite
  is called \emph{bounded} for $<_\xi$ if for every
  non-positive variable $x_i$, there exists an integer $r_i\geq 0$
  with $x_i^{r_i}\in I$.
\end{dfn}

\begin{lmm} \label{grobnerFamiliesAreBounded}
  Let  $I\in \mch ^{BB(\xi,\Delta)}(B)$ be an ideal. Then $I$ is bounded for $<_\xi$.
\end{lmm}
\begin{proof}
  Up to reordering the
  components of $\xi$, one can assume that $\xi_i\geq \xi_{i+1}$. We
  denote by $k$, $l$ the integers such that $\xi_i>0\Leftrightarrow i<k$
  and $\xi_i<0 \Leftrightarrow i\geq l$.


The monomials with exponents in
$\Delta$ form a basis of $B[\bfx]/I$ by Proposition
\ref{lm:cobase}. In particular, every
  monomial $\bfx^e$ leads to an element $\bfx^e+\sum_{m\in \Delta}c_m\bfx^m$
  in  $I$. For $i\geq l$, $r_i \gg 0$ and $\bfx^e=x_i^{r_i}$, we have
  $c_m=0$ since $\IN(I)=I^\Delta$. Thus $h_i:=x_i^{r_i}\in I$.

  Let $x_i$ be a non-positive variable of weight $0$.  To prove that
  $x_i^r \in I$ for large $r$, we shall prove that for any
  monomial $m\in B[x_l,\ldots,x_d]$, $x_i^r m\in I$ for large $r$,
  and subsequently apply that to the monomial $m=1$. If
  the exponent $(e_l,\dots,e_d)$
  of $m=x_l^{e_l}\cdots x_d^{e_d}$ is large (in concrete terms, if
  $e_i\geq r_i$ for some $i$), then $m\in I$ by the above. We are therefore left with a
  finite collection of $m$ for which the claim has to be checked. We proceed by
  induction over $m$. Let $m$ be a  minimal element of the finite
  collection for which the claim has not been
  proved yet.
  By our hypothesis on $I$, for large $r$, the monomial $x_i^r m$ lies in the limit ideal
  $I_\xi[0]$, thus $I$ contains an element $f=x_i^rm+R$,
  where all terms in $R$ are strictly smaller with respect to $<_\xi$ than $x_i^rm$.
  We write the monomials appearing in $R$ as $m_1m_2m_3$,
  with $m_1\in k[x_1\ldots,x_{k-1}],m_2\in k[x_k\ldots,x_{l-1}]$, $m_3\in
  k[x_l,\ldots,x_d]$. Such a product satisfies $m_1m_2m_3<_\xi x_i^r m$ only if
  $m_3<_\xi m$.
  By induction we know that when multiplying by an adequate power $x_i^p$, we get
  $x_i^pm_3\in I$, hence
  $x_i^{r+p}m=x_i^pf-x_i^pR \in I$, as required.
\end{proof}

\begin{pro}
\label{propSecondFinitnessLemma}
Let  $I\in \mch ^{BB(\xi,\Delta)}(B)$ be an ideal.
Then  for every non-positive variable $x_i$, we have $x_i^n \in I$, where $n:=\#\Delta$ is the cardinality of $\Delta$.
\end{pro}

\begin{proof}
  After applying a suitable permutation, we may assume that $x_1,\ldots,x_l$
  (resp. $x_{l+1},\ldots,x_d$) are the non-positive (resp. positive)
  variables, and we shall prove that $x_1^n \in I$.
  For every monomial $m \in k[x_2,\ldots,x_d]$, there exists a
  unique $h(m) \in \NN$, call it the \emph{height},
  such that $m^-:=x_1^{h(m)-1}m\in \Delta$ and $m^+:=x_1^{h(m)}m\notin
  \Delta$. In particular, $h(m)=0$ iff $m\notin \Delta$,
  and all heights of all $\bfx^e \in k[x_2,\ldots,x_d]$ sum up to $n$, the cardinality of $\Delta$.

  Let $M \subset k[x_2,\ldots,x_l]$ be the set of monomials with exponents in
  $\Delta' := \NN^{l-1} \cap \Delta$.
  The set $M$ is finite, and we number its elements
  such that $m_1^+\leq _\xi m_2^+ \leq _\xi \cdots \leq _\xi m_{\#M}^+$. For any monomial $m \in k[x_2,\ldots,x_d]$,
  we define $H(m):=\sum_{m_i^+\leq_\xi m^+} h(m_i)$.
  Note that  $H(m)\leq \sum_{\bfx^e \in
  k[x_2,\ldots,x_d]} h(\bfx^e)=n$. In particular, if we prove
  \begin{equation}\label{*}
    \forall m=\bfx^e\in k[x_2,\ldots,x_l],\ x_1^{H(m)}m\in I
  \end{equation}
  then for the particular choice $m:=1$, \eqref{*} implies $x_1^n \in I$, which concludes the proof.

  Since $I$ is bounded by Lemma \ref{grobnerFamiliesAreBounded},
  for each $i$ with $2\leq i \leq l$, there exists some $r_i$ such that $x_i^{r_i}\in I$.
  Thus \eqref{*} is true if $x_i^{r_i}$ divides $m$.
  It follows that the set $C:=\left\{m\in k[x_2,\ldots,x_l]\mid m \text{ is a monomial},\, x_1^{H(m)}m\notin I\right\}$ is
  finite. If \eqref{*} is not true, then $C\neq \emptyset$ and $C$ contains an
  element $m_0$ which is minimal in the sense that $m_0^+\leq_\xi m^+$ for every $m\in C, m
  \neq m_0$.

The decomposition of $m_0^+$ on the basis $\bfx^e,e\in \Delta$ of  $B[\bfx]/I$ yields an expression
$f=m_0^++\sum_{e\in\Delta}c_e\bfx^e$ with $f\in I$. By Proposition \ref{reformulationBB2Monic},
 $c_e=0$ if  $x^e\geq_\xi m_0^+$ . Thus
   \begin{equation}\label{**}
     f=m_0^++\sum_{\bfx^e<m_0^+}c_e \bfx^e.
   \end{equation}
We have
  \begin{displaymath}
    \bfx^e=x_1^{e_1}\cdots x_d^{e_d}\in \Delta \Rightarrow x_1^{e_1}\cdots
    x_l^{e_l}\in \Delta  \Rightarrow e_1<h(x_2^{e_2}\cdots x_l^{e_l}).
  \end{displaymath}
  Since $x_1$ is a non-positive variable, and $x_{l+1},\dots,x_d$ are
  positive, we get:
  \begin{displaymath}
    (x_2^{e_2}\cdots x_l^{e_l})^+=x_1^{h(x_2^{e_2}\cdots
    x_l^{e_l})}(x_2^{e_2}\cdots
    x_l^{e_l})\leq_\xi x_1^{e_1}x_2^{e_2}\cdots x_l^{e_l}x_{l+1}^{e_{l+1}}\cdots
    x_d^{e_d}=\bfx^e
  \end{displaymath}

  Here is the upshot of the above: if $C \neq \emptyset$, $m_0 \in C$ is its minimum,
  and a monomial $\bfx^e=x_1^{e_1}\cdots x_d^{e_d}$
  appears in a term of $f$, then $m_e := x_2^{e_2}\cdots x_l^{e_l}$
  satisfies $m_e^+\leq _\xi \bfx^e <_\xi m_0^+$. Thus $H(m_e)\leq H(m_0)-h(m_0)$,
  and by minimality of $m_0$, $m_ex_1^{H(m_e)}\in I$. It follows that
  the multiple $\bfx^ex_1^{H(m_0)-h(m_0)}$ lies in $I$. The product of the expression \eqref{**} with
  $x_1^{H(m_0)-h(m_0)}$ yields
  \begin{displaymath}
    m_0x_1^{H(m_0)}=m_0^+x_1^{H(m_0)-h(m_0)}=fx_1^{H(m_0)-h(m_0)}-\sum c_e\bfx^ex_1^{H(m_0)-h(m_0)} \in I,
  \end{displaymath}
  a contradiction. It follows that $C = \emptyset$, and \eqref{*} is true.
  \end{proof}

Using the order $<_\xi$, it not in general possible to make a division
like in Gr\"obner basis theory because the algorithm may not terminate,
due to the possible negative signs in the coordinates of
$\xi$. However, for Bia\l ynicki-Birula families, a substitute of a
division is possible. This proposition will not be used in the
sequel, but we include it for itself.
\begin{pro} \label{pseudo-division}
  Let $I\in \mch ^{BB(\xi,\Delta)}(B)$.
  Let $o_1,\ldots,o_u$ be
  the outside corners of $\Delta$, and $f_1,\dots,f_u\in I$ be elements
  with $\IN_{<_{\xi}}(f_i)=o_i$. Then for all $f\in B[\bfx]$,
  there exists a division
  \begin{displaymath}
    f=\sum \lambda_i f_i +R_\Delta + R'
  \end{displaymath}
  such that
  \begin{itemize}
  \item
  each term $\tau=b\bfx^e$ of $R_\Delta$ satisfies $e\in \Delta$,
  \item
  for every term $c\bfx^s$ of $R'$ and for every $m \in \Delta$,
  $\bfx^s<_\xi \bfx ^m$.
  \end{itemize}
 For every such division,
 \begin{itemize}
 \item $R_\Delta$ is independent of the choice of the division,
 \item $f\in I$ if, and only if, $R_\Delta=0$,
 \item the map $f\mapsto R_\Delta$ is a homomorphism of $B$-modules $B[\bfx] \to B[\bfx]/I$,
 where we identify the latter module with $B[\Delta] := \oplus_{e \in
   \Delta} B \bfx^e$
 using Lemma \ref{lm:cobase}.
 \end{itemize}
\end{pro}

\begin{proof}
  To construct the expected expression
  $f=\sum \lambda_i f_i +R_\Delta +R'$,
  we proceed in several steps.
  At each step $j$, we have an expression $f=\sum \lambda_{ij} f_i
  + T_j$. For $j=0$, we take $\lambda_{i0}:=0$, $T_0:=f$. We decompose $T_j=R_{\Delta,j} +R'_j$, with
  $R_{\Delta,j}:=\sum_{m\in \Delta}c_{mj}m$ and $R'_j:=\sum_{m\notin
    \Delta}c_{mj}m$.  The initial part $\IN_{<_{\xi}}(R'_j)$ contains
  a term
  $\mu_j \IN_{<_{\xi}}(f_{i_j})$ for some $f_{i_j}$ by hypothesis. We set
  $\lambda_{i_j(j+1)}:=\lambda_{i_jj}+\mu_j$ and
  $\lambda_{i'(j+1)}:=\lambda_{i'j}$ for $i'\neq i_j$. Then $T_{j+1}:= f-\sum
  \lambda_{i(j+1)} f_i=T_j-\mu_j f_{i_j}$ decomposes, analogously as above, into $T_{j+1}=R_{\Delta,j+1} +R'_{j+1}$.

  If, for some $j$, it happens that $R'_j=0$,
  then we define $R_\Delta:=R_{\Delta,j}$
  and $R':=0$, and have constructed the expected
  expression. Otherwise, after a finite number of steps,
  the terms in $R'_j$ are arbitrarily small and the condition
  of the second bullet is satisfied with $R'=R'_j$ and
  $R_\Delta=R_{\Delta,j}$.


  If $f\in I$, then in the expression $f=\sum \lambda_i f_i +R_\Delta +R'$, we have
  $R_\Delta=0$ since $\IN_{<_{\xi}}(f-\sum \lambda_i f_i)$ cannot lie in $\Delta$ by
  hypothesis. In particular, if $f=\sum \mu_i f_i +S_\Delta +S'$ is another division,
  we take their difference and obtain a division of $0\in I$, which implies
  $R_{\Delta}=S_\Delta$.

  It is obvious that $f\mapsto R_\Delta$ is a homomorphism of
  $B$-modules as it is possible
  to add divisions, or to multiply them with a scalar $\lambda\in
  B$.

  Let us now consider the $B$-submodule $B[\Delta]$ of $B[\bfx]$.
  The identity on
  $B[\Delta]$ factors as $B[\Delta]\rightarrow B[\bfx] \rightarrow
  B[\Delta]$ where the first arrow is the inclusion and the second is
  the morphism $R_\Delta$. The above implies that this factorization induces a factorization
  $B[\Delta]\rightarrow B[\bfx]/I \rightarrow B[\Delta]$. This composition
  is surjective between locally free modules of the same rank, so it is an
  isomorphism. In particular, we obtain that
  $R_\Delta=0$ implies $f\in I$.
\end{proof}

\section{Definition of monic functors}
\label{sec:base-changes-monic}

The goal of this section is to define monic functors, which
parameterize ideals with a prescribed initial ideal.  In general,
the initial ideal does not commute with arbitrary base
change, but only with flat base change
(see \cite{bayerGalligoStillman}). However,
we prove that  $\Delta$-monic families are
stable by arbitrary (non flat) base change (Proposition \ref{refinedStatementForMonicBaseChanges}) and
functoriality follows.

\begin{rmk}\label{rem:functorialityNeedsTotalOrder}
  When working with monic functors, we will consider total orders
  rather than the partial orders used in the previous section. The
  reason is that we control the base changes and the functoriality
  for the graded parts
  $\IN_m(I)$. With a total order, $\IN(I)=\bigoplus_{m\
    \textrm{monomial}}\IN_m(I)m$, thus we control the base change of the
  initial ideal through its graded pieces. The following
  proposition shows that this control may fail  for a partial order.
\end{rmk}
\begin{pro}
  Let $I\subset B[\bfx]$ be a $\Delta$-monic ideal for an order $<$
  refining the partial order $<_\xi$, ie. $\IN_{<,m}(I)=B$ if $m\notin
  \Delta$ and $\IN_{<,m}(I)=0$ if $m\in \Delta$. If $<$ is a total order, then
  $\IN(I)=I^{\Delta}$. This may not be true if $<$ is not a total
  order.
\end{pro}
\begin{proof}
  If the order is total, then the initial ideal is generated by
  terms thus it is graded by the degrees of the corresponding monomials.
  As for the counterexample, we take $I\subset k[x,y]$ generated by
  $<x+y,x^2,xy,y^2>$, $\Delta=\{1,x,y\},\xi=(1,1)$ and $<$ identical
  to $<_\xi$ (trivial refinement). Then the sum of the graded
  parts $\oplus_{m\ \mathrm{monomial}} \IN_m(I)m=I^\Delta$ but $\IN(x+y)=x+y\in
  \IN(I)$ is not in $I^\Delta$.
\end{proof}



\begin{pro} \label{refinedStatementForMonicBaseChanges}
  Let $<$ be a total order refining $<_\xi$. Let $I\subset B[\bfx]$
  be a bounded
  ideal. Let $m$ be a monomial.
 If for all $m'\geq m$, either $\IN_{m'}(I)=0$ or
  $\IN_{m'}(I)=\langle 1 \rangle$ holds,
  then for every base change $B\rightarrow A$, the equality $\IN_m(IA[\bfx])=\IN_m(I)A$ holds.
\end{pro}

\begin{proof}
  Obviously, $\IN_m(I)=\langle 1 \rangle$ implies $\IN_m(IA[\bfx])=\langle 1 \rangle$.
  If $\IN_m(I)=0$, we argue by contradiction, supposing that
  $\IN_m(IA[\bfx])\neq 0$. Choose $f\in IA[\bfx]$
  with $\IN(f)=am$, $a\neq 0$.  Let $x_j$ be a non-positive variable. Then
  $x_j^{d_j}\in I$ for $d_j$ large. In particular, replacing $f$ with
  $f-\lambda x_j^{e}$, we may assume that $f$ has no term divisible
  by $x_j^{d_j}$. Choose an expression
  \begin{equation}\label{combo}
    f=\sum a_if_i
  \end{equation}
  with $a_i\in
  A$ and $f_i\in I \subset B[\bfx]$. As above, we may assume that none of the $f_i$ contains
  a term divisible by $x_j^{d_j}$. Let $m'$ be the maximal monomial
  appearing in the $f_i$. Note that $m'\geq m$ and, more precisely, $m'>m$
  since $\IN_m(I)=0$.
  Among all possible expressions \eqref{combo}, choose one with minimal $m'$.
  (Although $<$ is not necessarily a monomial order, a minimal
  $m'$ exists since we have bounded the exponent of the non-positive
  variables.) Let $J:=\{i : \IN(f_i)=\lambda_i m'\}$ be the set of indices of $f_i$ with initial monomial $m'$.
  The coefficient of $m'$ vanishes on the right hand side of \eqref{combo}, thus
  $\sum_{i\in J} a_i \IN(f_i)=\sum_{i\in J}
  (a_i\lambda_i)m'=0$. Thus $\sum_{i\in J} a_i\lambda_i=0$. Let $g\in I$ with
  $\IN(g)=m'$ and such that $g$ has no term divisible
  by $x_j^{d_j}$. Let $f'_i:=f_i$ if $i\notin J$ and $f'_i:=f_i-\lambda_i
  g$ if $i\in J$. Then
  \begin{displaymath}
    f=\sum a_if'_i
  \end{displaymath}
  and this expression contradicts the maximality of $m'$.
\end{proof}

Let $X={\rm Spec}\,A$ be the affine scheme corresponding to the ring $A$.
The ideal $\IN_m(I) \subset A$ defines a closed subscheme of $X$.
If $X$ is a scheme which is not affine, we wish to
glue the local constructions we have been working with so far.
Since open immersions are flat, the following proposition
implies that gluing is possible and that for any sheaf of ideals $\mathcal{I} \subset \mathcal{O}_X[x]$,
there is a well-defined sheaf
of ideals $\IN_m(\mathcal{I}) \subset \mathcal{O}_X$ on a possibly non-affine scheme $X$.
This allows us to speak of \emph{bounded} and \emph{monic}, resp., \emph{ideal sheaves} rather than ideals.

\begin{pro}\label{baseChangeBGS}
  Let $<$ be a total order refining $<_\xi$. Let $I\subset B[\bfx]$
  be a bounded
  ideal.
  Let $B\rightarrow A$ be a ring homomorphism which makes $A$ a flat $B$-module. Then
  $\IN_m(IA[\bfx])=\IN_m(I)A$.
\end{pro}

\begin{proof}
  Theorem 3.6 of \cite{bayerGalligoStillman} proves the statement
  in the case where $<$ is a monomial order. The same proof also
  goes through in our context, provided that we take care of
  the high powers of the non-positive variables as we did in the
  proof of Proposition \ref{refinedStatementForMonicBaseChanges}.
\end{proof}


\begin{dfn} \label{defMonicFunctor}
  Let $<$ be a total order refining $<_\xi$ and $\Delta$ be a standard set. Let $B$ be a
  $k$-algebra and
  $\mch^{{\rm mon}(<,\Delta)}(B)$ be the set of ideals $I\subset B[\bfx]$ such
  that
  \begin{itemize}
  \item $I$ is bounded and $\Delta$-monic
  \item $ B[\bfx]/I$ is $B$ locally free of rank $\#\Delta$
  \end{itemize}
  This defines a covariant functor $\mch^{{\rm mon}(<,\Delta)}$,
  which we call a \emph{monic functor}, from the category
  of 
  $k$-algebras to the category of sets.
\end{dfn}

\begin{rmk} The above definition makes sense. Indeed,
it follows from the base change Proposition
\ref{refinedStatementForMonicBaseChanges}
that being a $\Delta$-monic ideal is stable by base change.
The boundedness property and the local freeness are also stable
by base change.
 \end{rmk}

\section{Monic functors are representable}
\label{sec:monic-functors-are}

The goal of this section is to prove that the monic functors are
representable (Theorem \ref{thm:theMonicFuctorIsrepresentable}).

\begin{pro}  \label{pro:universalityForMonicFamilies} Let $<$ be a total
  order refining $<_\xi$.
  Let $I\subset B[\bfx]$ be a bounded ideal with $B[\bfx]/I$ a locally
  free $B$-module of rank $n$,
  and let $\mathcal{I} \subset \mathcal{O}_X[\bfx]$ be the ideal sheaf on $X := {\rm Spec}\,B$ defined by $I$.
  Let $\Delta$ be a standard set of cardinality $n$. There exists a locally
  closed subscheme $Z \subset X$ such that
  \begin{itemize}
  \item the restriction of $\mathcal{I}$ to $Z$ is a bounded $\Delta$-monic family,
  \item any morphism $f:{\rm Spec}\,A \rightarrow {\rm Spec}\, B$ such that $IA[\bfx]$ is
    a bounded $\Delta$-monic family factors through $Z$.
  \end{itemize}
\end{pro}

Proposition \ref{pro:universalityForMonicFamilies}
is a particular case of Proposition
\ref{inductionForRepresentingMonicFamilies}, which we shall prove by induction.

\begin{lmm}\label{lm_hypecube}
   Let $I\subset B[\bfx]$ be a bounded ideal  for the order $<_\xi$
   such that $B[\bfx]/I$ is a finite $B$-module.
 Let $C:=\{x^e \mid e_i <r_i\}$ be a ``hypercuboid of monomials'' with
  edges of lengths $r_i$.
  Then it is possible to choose the  integers $r_1,\ldots,r_d$ such that:
  \begin{itemize}
  \item the monomials in $C$ generate  $B[\bfx]/I$ and $B[\bfx,t]/I_\xi[t]$
  \item if $m$ is a monomial with $m\notin C$, there exists $f\in I$  with
    $f=m+\sum_{e\in C} c_ex^e$, and  $\IN_{<_\xi}(f)=m$ .
  \end{itemize}
\end{lmm}
\begin{proof}
  Since $I$ is bounded, if $x_i$ is non-positive, we can choose $r_i$
  such that $x_i^{r_i}\in I$. Let $S$ be a set of monomials generators of
  $B[\bfx]/I$. We may suppose that  for every $m$ in $S$ and every non
  positive variable $x_i$, the exponent of $x_i$ in $m$ is  smaller than $r_i$.
  For $x_i$ positive, we can choose  $r_i$ large,  such that
  any monomial $m$ multiple of $x_i^{r_i}$ whose exponent in any
  non-positive variable $x_j$ is less than $r_j$ satisfies $m>_\xi s$
  for any $s\in S$. Moreover, we may choose $r_i$ larger than the
  exponent of $x_i$ in any monomial of $S$.
  The monomials in $C$ generate $B[\bfx]/I$ as
$C\supset S$.

  Let $m\notin C$ be a monomial.  Then $m$ is a multiple of
  some $x_i^{r_i}$. If one can take $x_i$ non-positive, $m\in I$  and we take $f=m$.
If not,  $x_i$ is positive and the decomposition in $B[\bfx]/I$ of  $m$
  on $S$ yields an expression $f=m-\sum _{s\in
    S}c_s\bfx^s$ in $I$ with $\IN_{<_\xi}(f)=m$.

  Then $I_\xi[t]$ contains the elements
  $t^{f_\xi(m)}(t\cdot f)=m-\sum _{s\in
    S}t^{f_\xi(m)-f_\xi(s)}c_s\bfx^s$. It
  follows that the quotient $B[t,\bfx]/I_\xi[t]$ is a finite
  $B[t]$-module generated by the monomials in $C$.
\end{proof}

\begin{pro}  \label{inductionForRepresentingMonicFamilies} Let $<$ be
  a total order refining $<_\xi$.
  Let $I\subset B[\bfx]$ be a bounded ideal 
  such that $B[\bfx]/I$ is a locally free $B$-module of rank $n$,
  and let $\mathcal{I} \subset \mathcal{O}_X[\bfx]$ be the ideal sheaf
  on $X := {\rm Spec}\,B$ defined by $I$.
  Let $C:=\{x^e \mid e_i <r_i\}$ be a ``hypercuboid of monomials'' of
  edge lengths $r_i$ and therefore, of cardinality $s:=\prod_{i=1}^d
  r_i$, satisfying the conditions of Lemma \ref{lm_hypecube}.
  We number the monomials $m_i\in C$ such that
  $m_1>m_2>\dots>m_s$. Let $r\leq s$, and fix a map
  $\mu:\{1,\ldots,r\}\rightarrow \{0,1\}$.
  Then there exists a locally
  closed subscheme $Z_r\subset X$ $($possibly empty$)$ such that
  \begin{itemize}
  \item The sheaf of ideals $\mathcal{I}_r$, which we define as the restriction of $\mathcal{I}$ to $Z_r$,
    is a bounded monic family
    with $\IN(\mathcal{I}_r)_{m}=\langle 1 \rangle$ for $m\notin C$ and
    $\IN(\mathcal{I}_r)_{m_i}=\langle \mu(i) \rangle$ for $1\leq i\leq r$.
  \item Let $f:{\rm Spec}\, A \rightarrow {\rm Spec}\, B$ be a morphism and
    $K:=IA[\bfx]$. Then $\IN(K)_{m}= \langle 1 \rangle$ for $m\notin C$ and
    we have $\IN( K)_{m_i}=\langle \mu(i)  \rangle$ for $1\leq i\leq r$ if, and only if,
    $f$ factors through $Z_r$.
  \end{itemize}
\end{pro}

\begin{proof}
  We start with the first item, which we prove by induction on $r\geq 0$.

  When $r=0$, we may take $Z_0=X$ since the condition
  $\IN(\mathcal{I}_r)_m=\langle 1 \rangle$ for $m\notin C$ is true by
  construction of $C$. We may assume that
  $Z_{r-1}$ is adequately defined.
  Let $F_r\subset Z_{r-1}$ the closed subscheme
  defined by the sheaf of ideals
  $\mathcal{I}(F_r):=\IN(\mathcal{I}_{r-1})_{m_r}$. Let $O_r:=Z_{r-1}\setminus F_r$.
  We define
  \begin{displaymath}
    Z_r:=\begin{cases}
      F_r & \text{ if } \mu(r)= 0 \\
      O_r & \text{ otherwise}.
    \end{cases}
  \end{displaymath}

  By Proposition \ref{refinedStatementForMonicBaseChanges} and the
  induction hypothesis, we have
  $\IN(\mathcal{I}_r)_{m}=\langle 1 \rangle$ for $m\notin C$ and
  $\IN(\mathcal{I}_r)_{m_i}=\langle \mu(i) \rangle$ for $1\leq i\leq r-1$.
  We have to prove that $\IN(\mathcal{I}_r)_{m_r}=\langle \mu(r) \rangle$.
  This is a local problem, so we may assume that $Z_{r-1} \subset X$ is a closed subscheme defined by an
  ideal $J_{r-1}\subset B$. If $\mu(r)=1$, the base change $Z_r
  \hookrightarrow Z_{r-1}$ is open, thus flat. Proposition \ref{baseChangeBGS}  therefore shows that
  $\IN(\mathcal{I}_{r})_{m_r}=\langle 1 \rangle$ on a neighborhood of any $p\in Z_r$.

  Assume now that $\mu(r)=0$.
  We claim that $\IN(\mathcal{I}_r)_{m_r}=0$. The problem is
  local, so we may assume that both $Z_r$ and $Z_{r-1}$ are affine.
  Accordingly, we replace the sheaves $\mathcal{I}_r$ and $\mathcal{I}_{r-1}$
  by their respective ideals of global sections,
  which we denote by $I_r$ and $I_{r-1}$, respectively.
  We will argue by contradiction, supposing that
  there exists some $f\in I_r$ with $\IN(f)=cm_r$, $c\neq 0$. Take some $g\in I_{r-1}$
  that restricts to $f$ over $Z_r$. Using Lemma \ref{lm_hypecube},
  we may assume that both $f$ and $g$
  are linear combinations of monomials in $C$, $f=\sum_{m_i\in C} a_i m_i$ and
  $g=\sum_{m_i\in C} b_i m_i$, resp.
  Among all possible $g$, choose one which minimizes $\IN(g)$. Then
  $\IN(g)=dm$ with $m\geq m_r$. Suppose that $m>m_r$. Since
  $\IN(I_{r-1})_m= \langle 0 \rangle$ or $\langle 1 \rangle$ by induction hypothesis and since
  $d\neq 0$, we obtain that $\IN(I_{r-1})_m= \langle 1 \rangle$. Choose  $h\in I_{r-1}$
  with $\IN(h)=m$. Then $g' := g-dh$
  contradicts the minimality of $g$. Thus $m=m_r$, and $d\in
  (I_{r-1})_{m_r}=I(F_r)$ vanishes on $F_r$. Since $d$ restricts to $c$ on
  $F_r$, it follows that $c=0$. This is a contradiction,
  which finishes the proof of $\IN(\mathcal{I}_r)_{m_r}=0$.

  We now come to the second point. If $f$ factors through $Z_r$,
  then $\IN(K)_{m}= \langle 1 \rangle$ for $m\notin C$ and
  $\IN(K)_{m_i}= \langle \mu(i) \rangle$ for $1\leq i\leq r$, as these properties
  are inherited from $Z_r$ by Proposition
  \ref{refinedStatementForMonicBaseChanges}.

  If, on the other hand, $f$ does not factor through $Z_r$,
  then we want to prove that $\IN(K)_{m_i}\neq \langle \mu(i) \rangle$ for some
  $i$, $1\leq i\leq r$. We may assume that $f$ factors through
  $Z_{r-1}$, since in the complementary case, we are done by induction.

  We first consider the case $\mu(r)=0$. By the factorization property of $f$
  through $Z_{r-1}$,
  we may assume that $B$ is the coordinate ring of the scheme $Z_{r-1}$. We denote by
  $f^\#:B\rightarrow A$ the morphism associated with $f$.
  Since $f$ does not factor through $Z_r$, there exists some $g\in
  I_{r-1}$ with $\IN(g)=cm_r$, $ f^\#(c)\neq 0$. The pullback of $g$ to
  $A[\bfx]$ shows that $\IN(IA[\bfx])_{m_r}\neq \langle 0 \rangle$, and we are done.

  Now consider the case $\mu(r)=1$. Since $Z_{r}\subset Z_{r-1}$ is open in
  this case, the factorization property of $f$ implies the existence of
  a point $p\in {\rm Spec}\,A$ with $f(p)\in Z_{r-1}$ and $f(p)\notin
  Z_r=Z_{r-1}\setminus F_r$. In particular, $\IN(I \cdot
  k(p)[\bfx])_{m_r}=0$ by Proposition   \ref{refinedStatementForMonicBaseChanges}.
  Thus $\IN(K)_{m_r}\neq \langle 1 \rangle$, as expected, since otherwise we would have
  $\IN(I \cdot k(p)[\bfx])_{m_r}= \langle 1 \rangle$ by
  Proposition \ref{refinedStatementForMonicBaseChanges}.
\end{proof}

\begin{rmk} It is natural to try to formulate the last proposition in
  terms of commutative algebra, without sheafs of ideals. However,
  the subscheme $Z_r$ is in general not an affine scheme. For
  instance, suppose that $k$ is an algebraically closed field,
  let $B=k[a,b]$, $I=I(a,b)=\langle x_1^3,x_2+ax_1+bx_1^2 \rangle$ and choose the
  order such that $x_1>>x_2$. Then for every closed point $(a_0,b_0)\neq (0,0)$, the ideal
  $I(a_0,b_0)$ contains an element with initial term $x_1^2$. Since
  $I(0,0)=\langle x_2,x_1^3 \rangle$, it follows that the locus where $I$ admits $x_1^2$
  as an initial term is the complement of the origin in the plane,
  which is not affine.
\end{rmk}

\begin{pro}\label{pro:uniformBoundingForMonic}
  Let $<$ be a total order refining $<_\xi$,
  Let $I\subset B[\bfx]$ be an  ideal in $\mch^{{\rm  mon}(<,\Delta)}(B)$, then for every non-positive variable $x_i$,
  we have $x_i^n\in I$ $($with $n=\#\Delta)$.
\end{pro}
\begin{proof}
  By definition, $I$ is bounded and $\IN_<(I)=I^\Delta$. Then we can run
  exactly the same proofs as in Lemma \ref{lm:cobase}
  and Proposition \ref{propSecondFinitnessLemma}. In these statements,
  the order is partial, but we can a fortiori consider the proof for a
  total order.
\end{proof}

\begin{thm}\label{thm:theMonicFuctorIsrepresentable}
  Let $<$ be a total order refining $<_\xi$.
  Then the monic functor $\mch^{{\rm mon}(<,\Delta)}$ is representable by a locally closed
  scheme $H^{{\rm mon}(<,\Delta)}$ of $H^n(\mathbb{A}^d)$, where $n=\#\Delta$.
\end{thm}
\begin{proof}
  Let $I\subset B[\bfx]$ be an ideal defining a flat family
  ${\rm Spec}\,B[\bfx]/I \to {\rm Spec}\,B$
  of relative length $n$ with
  $I$ bounded and  $\Delta$-monic. Let $L_i\subset H^n(\mathbb{A}^d)$
  be the closed subscheme of $H^n(\mathbb{A}^d)$ parametrizing the subschemes
  $Z$ included in the subscheme $\{x_i^n=0\} \subset \mathbb{A}^d$.
  Let $L := \cap_{x_i\text{ negative }}  L_i$. Since the universal ideal of the Hilbert
  scheme (see \cite{lederer} for the construction and properties of that universal ideal)
  is bounded over $L$, there is a locally closed subscheme $L_\Delta\subset L$
  parametrizing $\Delta$-monic ideals (Proposition
  \ref{pro:universalityForMonicFamilies}). By
  the universal property of the Hilbert scheme, the ideal $I$ corresponds to a unique morphism
  $\phi:{\rm Spec}\,B\to H^n(\mathbb{A}^d)$.
  Proposition \ref{pro:uniformBoundingForMonic} implies that
  the morphism $\phi$ factors through $L$.
  Proposition \ref{pro:universalityForMonicFamilies}
  (the universal property of $L_\Delta$)
  implies that $\phi$ even factors through $L_\Delta$. Conversely, any morphism
  ${\rm Spec}\,B\rightarrow L_\Delta$ yields a $\Delta$-monic bounded ideal by
  pullback of the universal ideal over the Hilbert scheme.
  Upon defining $H^{{\rm mon}(<,\Delta)}:=L_\Delta$, we thus get the required result.
\end{proof}

\section{Bia\l ynicki-Birula functors are representable}
\label{sec:bialyn-birula-funct-1}

The goal of this section is to prove that Bia\l ynicki-Birula functors
are representable.

So far, we have identified Bia\l ynicki-Birula families with families
  with an appropriate initial ideal for a \textit{partial} order. On
  the other hand, we have proved that functors of families with a
  prescribed initial ideal for a \textit{total} order are representable.
  To conclude, we will prove that having an initial ideal for a
  partial order is equivalent to having an initial ideal for two
  ad hoc total orders. The representability of the Bia\l ynicki-Birula
  functors will follow.

  The \emph{total} orders that we need are introduced in the
  following definition. We call them signed orders.
  They refine the partial orders $<_\xi$ used  in Section
  \ref{sec:bialyn-birula-funct}. Like the orders $<_\xi$, they are not
  monomial orders in the sense of \cite[Chapter~15]{eisenbud} if
  some components of $\xi$ are negative.



\begin{dfn}
 Suppose given  a map
  $\epsilon:\{1,\ldots,d\}\rightarrow \{-1,1\}$. The partial order $<_\xi$ on the monomials
  can be refined to a total order $<$ as follows: For all monomials  $\bfx^e$ and $\bfx^f$ with
    $f_\xi(e)=f_\xi(f)$, we have
    $\bfx^e<\bfx^f$ if, and only if,
    $(\epsilon(1)e_{1},\ldots,\epsilon(d)e_{d})<(\epsilon(1)f_{1},
    \ldots,\epsilon(d)f_{d})$
    in the lexicographic order.

  We call such an order $<$ a signed order refining $<_{\xi}$.
\end{dfn}

\begin{rmk}
Special cases of signed orders are implicit in    \cite{esBetti,esCells,grojnowski,nakajima},
for studying the Hilbert scheme of points in the two-dimensional
case. 
\end{rmk}

\begin{pro} \label{BBFunctor=intersectionOfTwoMonicFunctors}
  Let $\xi\in \ZZ^d$. Let $<_+$  be a signed order which refines
  $<_\xi$ and $\epsilon$ the function that defines $<_+$.
  Let $<_-$ be the ``opposite'' signed order, \textit{i.e.} the signed order
  defined by the opposite function $-\epsilon$.
  Then $\mch^{BB(\Delta,\xi)}=\mch^{{\rm mon}(<_+,\Delta)}\cap \mch^{{\rm mon}(<_-,\Delta)}$.
  \end{pro}

  \begin{proof}
  It is clear that $\mch^{BB(\Delta,\xi)}\subset \mch^{{\rm mon}(<_+,\Delta)}\cap
  \mch^{{\rm mon}(<_-,\Delta)}$ since by Proposition
  \ref{reformulationBB2Monic},
 $\mch^{BB(\Delta,\xi)}\subset
  \mch^{{\rm mon}(<,\Delta)}$ for any refinement $<$ of $<_\xi$.
  Conversely, take $I\in \mch^{{\rm mon}(<_+,\Delta)}(B) \cap
  \mch^{{\rm mon}(<_-,\Delta)}(B)$. For proving that $I\in
  \mch^{BB(\Delta,\xi)}(B)$, we first
  prove that $\IN_{<_\xi,_m}(I)=\langle 1 \rangle$ for $m\notin
  \Delta$. This will imply by Proposition \ref{limitAndOrder}
  that a monomial $m\notin \Delta$ satisfies $m\in I_\xi[0]$.

  We argue by contradiction. Suppose that the set $C:=\left\{m\notin \Delta\mid \IN_{<_\xi,m}(I)\neq \langle 1 \rangle\right\}$ is non empty. Since $I$ is
  bounded, $C$ is finite
  and  the quotient $B[t,\bfx]/I_\xi[t]$ is a finite $B[t]$-module,
  by Lemma \ref{lm_hypecube}. Let $m'$ be the smallest element of $C$ for the
  order $<_+$. Since $m'\notin \Delta$, there is an $f\in I$
  with $\IN_{<_+}(f)=m'$. We write $f=m'+r+s+t$, with
  $$r=\sum_{f_\xi(e)<f_\xi(m')} c_e\bfx^e,\quad s=\sum_{\substack{e\in \Delta\\ f_\xi(e)=f_\xi(m') \\ e<_+m'}} c_e\bfx^e,\quad\mathrm{and}\  t=\sum_{\substack{e\notin \Delta\\ f_\xi(e)=f_\xi(m')\\ e<_+m'}} c_e\bfx^e.$$
  By minimality of $m'$, for any term $c_e\bfx^e$ in $t$, there exists some $g_e \in I$ of shape
  $g_e=\bfx^e+\sum_{f_\xi(f)<f_\xi(e)}d_f\bfx^f$.

  Let
  \begin{displaymath}
    h:=f-\sum_{c_e\bfx^e \text{ a term of } t}c_eg_e.
  \end{displaymath}
  Then $h$ reads $h=m'+r+s+u$ with
  $$u=\sum_{\substack{e\notin \Delta\\ f_\xi(e)=f_\xi(m')\\ e<_+m'}} c_e(\bfx^e-g_e).$$
  The only terms $c_e\bfx^e$ in $h$ with $f_\xi(e)=f_\xi(m')$
  are the terms in $s$. Note that the orders $<_+$ and $<_-$ have
  been chosen such that any pair of monomials $m,m'$ satisfies the following:
  \begin{itemize}
    \item if $f_\xi(m)\neq f_\xi(m')$, then $m<_+m' \Leftrightarrow m<_-m'$
    \item if $f_\xi(m)=f_\xi(m')$, then $m<_+m' \Leftrightarrow m'<_-m$.
  \end{itemize}
  Thus, if $s\neq 0$, then $\IN_{<_-}(h)$ is a term $c_e\bfx^e$ with
  $e\in \Delta$, contradicting the assumption that $I\in\mch^{{\rm mon}(<_-,\Delta)}(B)$.
  We thus obtain that $s=0$ and $\IN_{<_\xi}(h)=m'$, contradicting the
  definition of $C$.

  Summing things up, ${\rm Spec}\, B[t,\bfx]/I_\xi[t]$ is a finite family over ${\rm Spec}\,
  B[t]$. By construction, this is a flat family of relative length $n$ over the
  open set $t\neq 0$. The fiber $ B[\bfx]/I_\xi[0]$ over $t=0$ is a
  quotient of $ B[\bfx]/I^\Delta$ which is a flat family of
  relative length $n$ over ${\rm Spec}\, B$. It follows by semi-continuity that
  $B[t,\bfx]/I_\xi[t]$ is a locally free $B[t]$-module of rank $n$ with
  $I_\xi[0]=I^\Delta$. Indeed, let $\mathrm{Fitt}_i\subset B[t]$ be the $\supth{i}$
  Fitting ideal of $B[t,\bfx]/I_\xi[t]$. The local freeness of rank
  $n$ of $B[t,\bfx]/I_\xi[t]$ is equivalent to the equalities
  $\mathrm{Fitt}_{n}=B[t]$ and $\mathrm{Fitt}_{n-1}=0$. The equality  $\mathrm{Fitt}_{n}=B[t]$
  is true since the fiber of the family
  has length at most $n$ over each closed point. Since the Fitting
  ideals commute with the base change $B[t]\rightarrow B[t,t^{-1}]$,
  we obtain $\mathrm{Fitt}_{n-1}B[t,t^{-1}]=0$, hence $\mathrm{Fitt}_{n-1}=0$
and $B[t,x]/I_\xi[t]$ has rank $n$.  Since
  a surjective morphism between free modules of the same rank is an
  isomorphism, the equality  $I_\xi[0]=I^\Delta$ follows.
\end{proof}

\begin{thm} \label{thm:BBrepresentable}
  The Bia\l ynicki-Birula
functor  $\mch^{BB(\Delta,\xi)}$
is representable.
\end{thm}

\begin{proof}
  By the above, the Bia\l ynicki-Birula  functor is an intersection of two functors
  $\mch^{{\rm mon}(<_+,\Delta)}$ and $\mch^{{\rm mon}(<_-,\Delta)}$,
  both representable by locally closed subschemes
  $H^{{\rm mon}(<_+,\Delta)}$ and $H^{{\rm mon}(<_-,\Delta)}$, respectively,
  of the Hilbert scheme. The Bia\l ynicki-Birula functor is therefore
  representable by the schematic intersection $ H^{{\rm mon}(<_+,\Delta)}\cap
  H^{{\rm mon}(<_-,\Delta)}$.
\end{proof}

\subsection{Signed orders are limit orders}
\label{sec:signed-orders-are}
The results of this section are not directly used in the proof of our
main results. However, they give some intuition on the role of
the signed orders which may appear unnatural at first glance:
a signed order $<$ is obtained from $<_\xi$ with
an infinitesimal deformation of $\xi$,
and the monic functors are stable under infinitesimal deformations.
These facts explain our strategy to replace $<_\xi$ with a signed order $<$
when studying the representability of monic functors.


  \begin{dfn}
  \leavevmode
  \begin{itemize}
    \item A sequence of partial orders $<_j$ \emph{converges to the total order $<$} if
    for every pair of monomials $a,b$, we have $a<b$ if, and only if, $a<_j b$ for $j$ large enough.
    \item Let $<$ be a signed order defined by the function $\epsilon$.  A
      \emph{sequence compatible with $\epsilon$ }
    is a sequence $\xi^j$ in $\RR^d$ converging to $0$ such that the sign of
    $\xi^j_k$ is $\epsilon(k)$ and such that the quotient
    $\xi^j_l/\xi^j_k$ tends to $0$ if, and only if, $k<l$.
  \end{itemize}
  \end{dfn}


The connection between signed orders and convergence is given by the next proposition (the proof of which, being easy, is left to the reader).

  \begin{pro}
    Let $<$ be  a refinement of the order
    $<_\xi$. Then the following
    conditions are equivalent:
    \begin{itemize}
    \item The order $<$ is the signed order defined by the function $\epsilon$
    \item For every sequence $\xi^j$ compatible with $\epsilon$,
      the sequence of orders $<_{\xi+\xi^j}$ converges to $<$.
    \end{itemize}
  \end{pro}

  The following proposition tells that
  monic functors are stable under small
  deformations of the order as long as
  we consider families parameterized by noetherian rings.

\begin{pro} \label{limitsForMonomialFunctors}
  Let $<,<_j$ be total orders refining $<_\xi$. We suppose  that the
  sequence of orders $<_j$ converges
  to  $<$. Then for $j$ large, the functors
  $\mch^{\rm mon(<_j,\Delta)}_{\mathrm{noeth}}$ and $\mch^{\rm mon(<,\Delta)}_{\mathrm{noeth}}$ are isomorphic.
\end{pro}
  \begin{proof}
    This is Proposition 22 in the first arxiv version of this paper \cite{firstVersion}, where
    noetherianity is always assumed. The proof is thus omitted.
  \end{proof}

\section{Bia\l ynicki-Birula schemes and Hilbert-Chow morphisms}
\label{sec:bialyn-birula-strata}

The goal of this section is to prove the following theorem:

\begin{thm} \label{thm:HilbertChow}
  If $\xi_i\leq 0$, then $H^{BB(\xi,\Delta)}$ is schematically
  included in the fiber over the origin $\rho_i^{-1}(0)$, where
   $\rho_i: H^n(\mathbb{A}^d)\rightarrow {\rm Sym}^n(\mathbb{A}^1)$
  is the Hilbert-Chow morphism associated with the $\supth{i}$
  coordinate.
\end{thm}

To be precise, let $\rho:H^n(\mathbb{A}^d)\rightarrow {\rm Sym}^n(\mathbb{A}^d)$ be the
usual Hilbert-Chow morphism. The projection $p_i:\mathbb{A}^d \rightarrow
\mathbb{A}^1$ to the $\supth{i}$-coordinate
induces a morphism $p_i^n:{\rm Sym}^n(\mathbb{A}^d)
\rightarrow {\rm Sym}^n(\mathbb{A}^1)$. Recall that ${\rm
  Sym}^n(\mathbb{A}^1)\simeq \mathbb{A}^n$ via elementary symmetric
functions. We denote by $\rho_i:=p_i^n\circ
\rho$ the Hilbert-Chow morphism associated with the $\supth{i}$
coordinate. We denote by $0\in {\rm Sym}^n(\mathbb{A}^1)$ the point
corresponding to $n$ copies of the origin of $\mathbb{A}^1$.

\begin{lmm}
  Let $I\in \mch^{{\rm BB}(\xi,\Delta)}(B)$. Suppose that $x_i$ is a
  non-positive variable. Let $m_i$ be the multiplication by $x_i$ in $B[\bfx]/I$.
  Then there exists a basis of $B[\bfx]/I$ such that the matrix of $m_i$ is
  strictly lower triangular.
\end{lmm}

\begin{proof}
  We consider any signed order $<$ refining $<_\xi$
  and defined by $(\epsilon,o)$ with $o=\mathrm{Identity}$ and $\epsilon(i)=-1$.
  The monomials $b_i$ with exponents in $\Delta$ are a basis of $B[\bfx]/I$.
  We number them such that $b_1>b_2>\cdots >b_n$.
  Then, if $x_i b_j\in \Delta$, we get $x_ib_j=b_l$, $l>j$.
  If $x_ib_j \notin \Delta$.
  The decomposition of
  $x_ib_j\in B[\bfx]/I$ yields $x_ib_j=\sum_{b_k \in
    \Delta}c_kb_k\ \mod I$. Since $\IN_{<_\xi}(I)=I^\Delta$, $c_k\neq 0$
  implies $b_k<_\xi x_ib_j\leq_\xi b_j$ and $k>j$.
\end{proof}

\begin{proof}[Proof of Theorem \ref{thm:HilbertChow}]

  We recall the observation by Bertin \cite{bertin} that the Hilbert-Chow
  morphism is given by the linearized determinant of Iversen. Let
  $I\in \mch^{{\rm mon}(<,\Delta)}(B)$ and $b_1,\dots, b_n$ a basis
  of $B[\bfx]/I$. If $P\in k[x_i]$, we denote by $C_P^j$ the $j^{th}$
  column of the matrix (with respect to our fixed basis) of multiplication by $P$. If
  ${P_1}\otimes \cdots \otimes {P_n}$ is a pure tensor in
  $k[x_i]^{\otimes n}$, we put ${\rm ld}(P_1\otimes \cdots \otimes
  P_n) := \det(C_{P_1}^1,\ldots,C_{P_n}^n)$. The symmetric group $S_n$ acts
  on $k[x_i]^{\otimes n}$; we denote by $k[x_i]^{(n)}\subset
  k[x_i]^{\otimes n}$ the invariant part. Iversen \cite{iversen} proved that
  ${\rm ld}:k[x_i]^{(n)}\rightarrow B$ is a $k$-algebra homomorphism. As
  was remarked by Bertin, this homomorphism corresponds to the
  Hilbert-Chow morphism $\rho_i$.
  The ideal of the origin is generated by the elementary symmetric
  polynomials, which have degree at least one. For
  proving the theorem, it therefore suffices to show that
  $\det(C_{P_1}^1,\dots,C_{P_n}^n)=0$ if $x_i$ divides some $P_j$.
  According to the lemma above, that determinant is the determinant of
  a lower triangular matrix, and that triangular matrix has a zero
  term on the diagonal if $x_i$ divides some $P_j$.
\end{proof}

\section{Relations to literature}
\label{sec:relationsToLitterature}

After the first version of this paper was posted on the arXiv, Drinfeld
has proved in \cite{drinfeld} results that allow
to retrieve several results of this paper, with
a different approach.

Drinfeld settles his constructions
in the category of algebraic spaces of finite type over a
field. Using Drinfelds's
language (\S\,0.2, \emph{ibid.}), let $H^+$ be the attractor of the Hilbert
scheme of points $H=H^n(\mathbb{A}^d)$, then $H^+$ is represented by a finite type
scheme, and the limit map $q^+:H^+\rightarrow H^T$ is affine (Theorem~1.4.2,
\emph{ibid.}). Theorem~1.4.2 deals with algebraic spaces, but since $H^T$ is a scheme and $q^+$ is  affine, it follows
that $H^+$ is a scheme as well. The functor $H^{BB(\Delta,\xi)}$ is canonically identified with the
fiber $(q^+)^{-1}([Z^{\Delta}])$, hence is representable by an affine
scheme.

The Hilbert-Chow morphism is considered in Theorem
\ref{thm:HilbertChow} under the condition $\xi_i\leq 0$.
In the special case $\xi_i<0$, Drinfeld's results can be applied :
$(\mathrm{Sym}^n(\mathbb{A}^1))^+=\{0\}$ by \cite[\S\,1.3.4]{drinfeld}; it follows by functoriality of
$(-)^+$ that the whole $H^+$ is schematically included in the fiber over
$\{0\}$.

To get the locally closed embedding, note that $H$ is
quasi-projective and fix a $T$-equivariant locally closed embedding
$H \subset \mathbb{P}^N$. Then $(H)^+ \subset (\mathbb{P}^N)^+$
is also a locally closed embedding by Lemmas~1.4.7 and~1.4.9 from \cite{drinfeld}. Moreover, by
the classical Bia\l ynicki-Birula decomposition, each component of
$(\mathbb{P}^N)^+$ is locally closed in $\mathbb{P}^N$. Putting everything
together, a component of
$H^+$ is a locally closed subscheme of $\mathbb P^N$ that
factors through the locally closed $H$. It follows that
$H^{BB(\Delta,\xi)}\rightarrow H$ is a locally closed embedding.
\medskip

We conclude by proving that some Bia\l ynicki-Birula strata
are reducible for the Hilbert scheme of
$\mathbb{A}^3$.
It is a consequence of the reducibility of the Hilbert scheme
proved by Iarrobino in \cite{iarro}.
\begin{pro}\label{BBmayBeIrreducible}
  Let $n$ be such that the Hilbert scheme $H^n(\mathbb{A}^3) $ is reducible.
  Let   $n_1>>n_2>>n_3>>0$  and let $T$ be the one-dimensional torus acting on
  $\mathbb{A}^3$ via $t\cdot(x,y,z)=(t^{n_1}x,t^{n_2}y,t^{n_3}z)$. Then there exists $\Delta$ such that
  the Bia\l ynicki-Birula stratum $H^{BB(T,\Delta)}\subset  H^n(\mathbb{A}^3) $ is reducible.
\end{pro}
\begin{proof}
  Let $C_1\subset H^n(\mathbb{A}^3) $ be the component
  containing the unions of $n$ distinct
  points and $C_2\subset H^n(\mathbb{A}^3) $ be any other
  component.
  By the choice of the weights, the fixed points for the torus
  action are the monomial schemes $Z^\Delta$ with ideal $I^\Delta$. Thus,
  if $Z \in C_2\setminus C_1$,
  $\lim_{t\to 0}t\cdot Z=Z^{\Delta}$ for some
  $\Delta$. In particular,  $\dim(C_2\cap H^{BB(T,\Delta)})\geq 1$.

  We have $\dim (C_1\cap H^{BB(T,\Delta)})\geq 1$ too.
 Indeed, let $Z$ be the
  disjoint union  of points with integer coordinates $(i,j,k)$ with
  $(i,j,k)\in \Delta$. If $x^ay^bz^c \in I^\Delta$, any point
  in the support of $Z$ has coordinates  $0\leq x<a$ or $0\leq y<b$ or
  $0 \leq z<c$,
  thus $f=x(x-1)\ldots (x-(a-1))y(y-1)\ldots
  (y-(b-1))z(z-1)\ldots(z-(c-1))$ is in $I(Z)$. We have $\lim_{t \to
    0}t^{n_1a+n_2b+n_3c}(t\cdot f)=x^ay^bz^c $, thus
  $\lim_{t \to 0}t\cdot Z\subset Z^{\Delta}$. The equality  $\lim_{t \to 0}t\cdot Z=
  Z^{\Delta}$ follows since the two schemes have the same length.

  Thus
  $H^{BB(T,\Delta)}$ contains at least two components, one in $C_1$
  and one in $C_2$, meeting in $Z^\Delta$.
\end{proof}

\end{document}